\documentclass[11pt]{amsart}

\usepackage{amssymb}
\usepackage{amsmath,tikz,hyperref}
\usepackage{amsthm}
\usepackage{eucal,mathrsfs}

\newtheorem{theorem}{Theorem}

\newtheorem{proposition}[theorem]{Proposition}
\newtheorem{corollary}[theorem]{Corollary}

\theoremstyle{definition}
\newtheorem{definition}[theorem]{Definition}
\newtheorem{example}[theorem]{Example}

\def\A{\mathcal{A}}
\def\D{\mathcal{D}}
\def\L{\mathcal{L}}
\def\U{\mathcal{U}}
\def\V{\mathcal{V}}

\def\Dyck{\mathscr{D}}

\title{On factor-free Dyck words with half-integer slope}

\author{Daniel Birmajer}
\address{Nazareth College, 4245 East Ave., Rochester, NY 14618}

\author{Juan B. Gil}
\address{Penn State Altoona, 3000 Ivyside Park, Altoona, PA 16601}
\author{Michael D. Weiner}

\begin{document}
\maketitle

\begin{abstract}
We study a class of rational Dyck paths with slope $\frac{2m+1}{2}$ corresponding to factor-free Dyck words, as introduced by P.~Duchon. We show that, for the slopes considered in this paper, the language of factor-free Dyck words is generated by an auxiliary language that we examine from the algebraic and combinatorial points of view. We provide a lattice path description of this language, and give an explicit enumeration formula in terms of partial Bell polynomials. As a corollary, we obtain new formulas for the number of associated factor-free generalized Dyck words.
\end{abstract}

\section{Introduction}
\label{sec:intro}

In these notes we consider the set of factor-free words belonging to the generalized Dyck language constructed from the alphabet $\A=\{a, b\}$ with valuations $h(a) = 2m+1$, $m\in\mathbb{N}$, and $h(b)=-2$. This is an instance of the two-letter Dyck language studied by P.~Duchon for which words correspond to Dyck paths with rational slope. In the case at hand, the slope is $\frac{2m+1}{2}$. 

In \cite{Duchon}, Duchon provided an algebraic grammar for generalized Dyck languages (as introduced by Labelle and Yeh \cite{LaYe90}) and proved that words in such a language can be obtained uniquely by inserting words of the language into factor-free words of the same language. 

Recall that a \emph{factor} of a word $w$ is any word $w^\prime$ such that $w = w_1 w^\prime w_2$. If $w_1$ and $w_2$ are not both the empty word, $w^\prime$ is a \emph{proper factor} of $w$. For a given alphabet $\A$ we let $\A^*$ denote the set of all words made from $\A$, together with the empty word $\varepsilon$. A word $w$ in $\A^*$  is said to be a \emph{generalized Dyck word} if it satisfies the conditions that $h(w) = 0$, and for each left factor $w_1$ of $w$, $h(w_1) \ge 0$. We denote by $\Dyck_{\A,h}$ the set of generalized Dyck words over the alphabet $\A$ with valuation given by $h$.  Moreover, we say that a word in $\A^*$ is $\Dyck_{\A,h}$-factor-free (or simply factor-free if the underlying Dyck language is clear) if it has no proper factor belonging to $\Dyck_{\A,h}$. The set of factor-free words in $\Dyck_{\A,h}$ will be denoted by $D_{\A,h}$.

As shown in \cite[Section 5]{Duchon}, the algebraic grammars for $\Dyck_{\A,h}$ and $D_{\A,h}$ can be described by a system of derivation rules in terms of certain auxiliary languages with restrictions on their total and partial valuations. As we will see in Section~\ref{sec:preliminaries}, for Dyck words with slope $\frac{2m+1}{2}$, the aforementioned derivation rules may be reduced to a single core language that we denote by $U$, or $U_{\frac{2m+1}{2}}$ if we wish to emphasize the slope.

The main focus of this paper is to study the auxiliary language $U$ from the algebraic and combinatorial points of view. We provide a description of $U$ in terms of lattice paths and, based on a polynomial equation satisfied by the generating function, we give an explicit enumeration formula involving partial Bell polynomials. As a corollary, we obtain new formulas for the enumeration of the corresponding factor-free generalized Dyck words. In Section \ref{sec:examples} we illustrate our results for slopes $\frac32$ and $\frac52$. In particular, we discuss a bijection between the elements of $U_{\frac52}$ and certain colored trees having nonleaf nodes of outdegrees 2 or 4. In the last section of the paper, we provide the building blocks needed to create factor-free words with slope $\frac72$, we give an interesting connection between the auxiliary language $U$ and certain colored Dyck paths, and we briefly discuss the use of factor-free Dyck words to generate cross-bifix-free (non-overlapping) codes of binary words with variable length.
 
\section{Grammar for factor-free words and the auxiliary language \texorpdfstring{$U$}{U}}
\label{sec:preliminaries}

In this section, we will review some of the terminology introduced by Duchon in \cite{Duchon} and will discuss the reduced algebraic grammar for the two-letter sublanguage of factor-free Dyck words with slope $\frac{2m+1}{2}$. In addition, we define the language $U$ alluded to in the introduction and describe its elements in terms of lattice paths. 
 
First, consider the auxiliary languages $L_i$ and $R_j$ defined as follows: 
\begin{itemize}
  \item $L_i$ is the set of factor-free words $w \in \A^*$ with total valuation $i$, such that each nonempty left factor $w_1$ of $w$ has $h(w_1) > i$,
  \item $R_j$ is the set of factor-free words $w \in \A^*$ with total valuation $-j$, such that each nonempty left factor $w_1$ of $w$ has $h(w_1) > 0$.
\end{itemize}

Using Duchon's results \cite[Section 5]{Duchon}, we conclude that if $\A=\{a, b\}$ is a two-letter alphabet with valuations $h(a) = 2m+1$ and $h(b)=-2$, then the set $D=D_{\A,h}$ of factor-free Dyck words in $\A^*$ may be described via the derivation rules:
\begin{gather*}
  D = \varepsilon + L_1R_1 + L_2R_2,\\
  L_i =  L_{i+1}R_1 + L_{i+2} R_2 \;\text{ for } 1\le i \le 2m-1,\\
  L_{2m} = L_{2m+1}R_1, \quad L_{2m+1} = a, \\
  R_1 = L_1R_2, \quad R_2 = b,
\end{gather*}  
which can be reduced to
\begin{equation}\label{eq:LR}
\begin{gathered}
  D = \varepsilon + L_1L_1b + L_2b,\\
  L_{2m+1} = a, \; L_{2m} = aL_1b, \\
  L_i =  L_{i+1}L_1b + L_{i+2}b \:\text{ for } 1\le i \le 2m-1. 
\end{gathered}  
\end{equation}
This is an unambiguous context-free grammar.

\begin{definition}
Let $U=U_{\frac{2m+1}{2}}$ be the set consisting of the empty word $\varepsilon$ together with all factor-free words $w \in \A^*$ with total valuation $0$, having at least one left factor with negative valuation, and such that each left factor $w_1$ of $w$ has $h(w_1) > -2m$.
\end{definition}

By identifying the letter $a$ with an east-step $(1,0)$ and the letter $b$ with a north-step $(0,1)$, each nonempty word in $U$ corresponds to an east-north lattice path from $(0,0)$ to the line $y=\frac{2m+1}{2}x$ with the following properties:
\begin{itemize}
\itemsep0pt
\item[$\triangleright$] it has no two lattice points on a line with slope $\frac{2m+1}{2}$ such that the path connecting them lies completely below that line (factor-free),
\item[$\triangleright$] it crosses the line $y=\frac{2m+1}{2}x$ at least once,
\item[$\triangleright$] it stays strongly below the line $y=\frac{2m+1}{2}x+m$.
\end{itemize}

Observe that, factoring $L_1=aUb^m$, the derivation rules \eqref{eq:LR} give
\begin{equation}\label{eq:U}
\begin{gathered}
  D = \varepsilon + aUb^m aUb^{m+1} + L_2b, \\
  L_{2m+1} = a, \; L_{2m} = aaUb^{m+1}, \\ 
  L_i =  L_{i+1} aUb^{m+1} + L_{i+2}b \:\text{ for } 1\le i \le 2m-1. 
\end{gathered}  
\end{equation}

In other words, the set of factor-free Dyck words with slope $\frac{2m+1}{2}$ is completely determined by the auxiliary language $U$. Note that the length of a word in $U$ is necessarily a multiple of $2m+3$.

\begin{example}[slope $\frac32$]
If $m=1$, then the derivation rules \eqref{eq:U} become
\begin{gather*}
  D = \varepsilon + aUbaUbb + L_2b, \\
  L_{3} = a, \; L_{2} = aaUbb, \\ 
  L_1 =  L_{2} aUbb + L_{3}b,
\end{gather*}  
which yield the equations
\begin{equation} \label{eq:U_recurrence32}
\begin{gathered} 
 D = \varepsilon + aUbaUbb + aaUbbb, \\
 U = \varepsilon + aUbbaUb.
\end{gathered}
\end{equation}
In particular, $abbab$ is the only word in $U$ of length 5, and the words in $U$ of length 10 are obtained by inserting $abbab$ into itself according to \eqref{eq:U_recurrence32}. Thus we obtain the two words $a{\color{blue}abbab}bbab$ and $abba{\color{blue}abbab}b$, see Figure~\ref{fig:figure1}.

\medskip
\begin{figure}[!ht]
\begin{center}
\begin{tikzpicture}[scale=0.5]
\begin{scope}
\draw [step=1,thin,gray!40] (0,0) grid (2,3);
\draw [dotted,thick] (0,1) -- (1.3,1.3*1.5+1);
\draw [gray!60, thick] (0,0) -- (2,3);
\draw [very thick] (0,0) -- (1,0) -- (1,2) -- (2,2) -- (2,3);
\end{scope}
\begin{scope}[xshift=40mm]
\draw [step=1,thin,gray!40] (0,0) grid (4,6);
\draw [dotted,thick] (0,1) -- (3.3,3.3*1.5+1);
\draw [gray!60, thick] (0,0) -- (4,6);
\draw [very thick] (0,0) -- (2,0) -- (2,2) -- (3,2) -- (3,5) -- (4,5) -- (4,6);
\draw [blue!80,very thick] (1,0) -- (2,0) -- (2,2) -- (3,2) -- (3,3);
\end{scope}
\begin{scope}[xshift=100mm]
\draw [step=1,thin,gray!40] (0,0) grid (4,6);
\draw [dotted,thick] (0,1) -- (3.3,3.3*1.5+1);
\draw [gray!60, thick] (0,0) -- (4,6);
\draw [very thick] (0,0) -- (1,0) -- (1,2) -- (3,2) -- (3,4) -- (4,4) -- (4,6);
\draw [blue!80,very thick] (2,2) -- (3,2) -- (3,4) -- (4,4) -- (4,5);
\end{scope}
\end{tikzpicture}
\caption{Lattice paths corresponding to $abbab$, $a{\color{blue}abbab}bbab$, $abba{\color{blue}abbab}b \in U_{\frac32}$.}
\label{fig:figure1}
\end{center}
\end{figure}
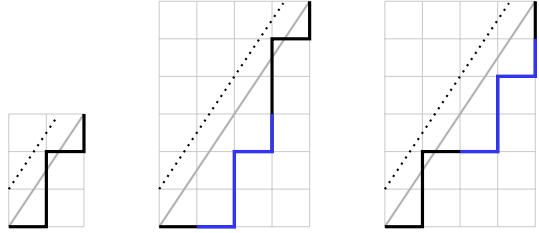
As already pointed out in \cite{Duchon}, in the case of slope $\frac32$ the elements of $U$ are enumerated by the Catalan numbers.
\end{example}

\begin{example}[slope $\frac52$] \label{ex:slope5/2}
If $m=2$, then \eqref{eq:U} gives
\begin{gather*}
  D = \varepsilon + aUbb aUbbb + L_2b, \\
  L_{5} = a, \; L_{4} = aaUbbb, \; L_3 =  L_{4} aUbbb + L_{5}b, \\ 
  L_2 =  L_{3} aUbbb + L_{4}b, \; L_1 =  L_{2} aUbbb + L_{3}b,
\end{gather*}  
which yield
\begin{equation} \label{eq:U_recurrence52}
\begin{gathered}
  D = \varepsilon + aUbb aUbbb + a(aUbbb)^3 b + abaUbbbb + aaUbbbbb, \\
  U = \varepsilon + (aUbbb)^3aUb + baUbbbaUb + aUbbbbaUb + aUbbbaUbb.
\end{gathered}
\end{equation}
In particular, $babbbab$, $abbbbab$, and $abbbabb$ are the only words of length 7 in $U$, and $abbbabbbabbbab$ is the only word of length 14 in $U$ that cannot be derived from the three primitive words of length 7, see Figure~\ref{fig:figure2}.
\begin{figure}[!ht]
\begin{center}
\begin{tikzpicture}[scale=0.5]
\begin{scope}[yshift=15mm]
\draw [step=1,thin,gray!40] (0,0) grid (2,5);
\draw [dotted,thick] (0,2) -- (1.2,1.2*2.5+2);
\draw [gray!60, thick] (0,0) -- (2,5);
\draw [very thick] (0,0) -- (0,1) -- (1,1) -- (1,4) -- (2,4) -- (2,5);
\end{scope}
\begin{scope}[xshift=40mm,yshift=15mm]
\draw [step=1,thin,gray!40] (0,0) grid (2,5);
\draw [dotted,thick] (0,2) -- (1.2,1.2*2.5+2);
\draw [gray!60, thick] (0,0) -- (2,5);
\draw [very thick] (0,0) -- (1,0) -- (1,4) -- (2,4) -- (2,5);
\end{scope}
\begin{scope}[xshift=80mm,yshift=15mm]
\draw [step=1,thin,gray!40] (0,0) grid (2,5);
\draw [dotted,thick] (0,2) -- (1.2,1.2*2.5+2);
\draw [gray!60, thick] (0,0) -- (2,5);
\draw [very thick] (0,0) -- (1,0) -- (1,3) -- (2,3) -- (2,5);
\end{scope}
\begin{scope}[xshift=120mm,scale=0.9]
\draw [step=1,thin,gray!40] (0,0) grid (4,10);
\draw [dotted,thick] (0,2) -- (3.21,3.21*2.5+2);
\draw [gray!60, thick] (0,0) -- (4,10);
\draw [very thick] (0,0) -- (1,0) -- (1,3) -- (2,3) -- (2,6) -- (3,6) -- (3,9) -- (4,9) -- (4,10);
\end{scope}
\end{tikzpicture}
\caption{Lattice paths for $babbbab$, $abbbbab$, $abbbabb$, $abbbabbbabbbab \in U_{\frac52}$.}
\label{fig:figure2}
\end{center}
\end{figure}
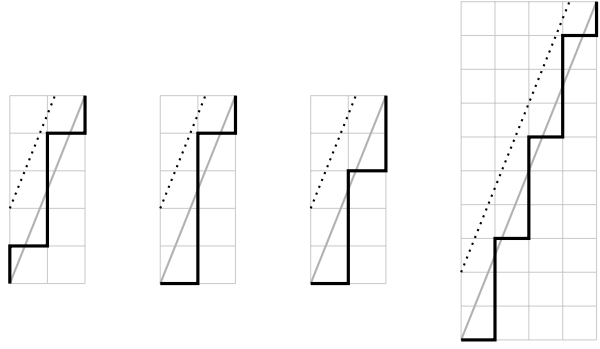

For some enumerative and asymptotic results for rational Dyck paths with slope $\frac52$, see \cite{BanWal15}.
\end{example}

\section{Enumeration of \texorpdfstring{$U$}{U} and factor-free \texorpdfstring{$\frac{2m+1}{2}$}{(2m+1)/2}-Dyck words}
\label{sec:enumeration}
  
In this section we prove a polynomial equation satisfied by the generating function $\U(\tau)$ of $U$, and use it to find an enumeration formula in terms of partial Bell polynomials. As a corollary, we obtain formulas for the enumeration of the corresponding factor-free generalized Dyck words.

Recall that the auxiliary languages $L_i$ satisfy:
\begin{gather*}
  L_{2m+1} = a, \; L_{2m} = aL_1b, \\
  L_i =  L_{i+1}L_1b + L_{i+2}b \:\text{ for } 1\le i \le 2m-1.
\end{gather*}  
This is an unambiguous grammar, so the above equations map to algebraic equations for the corresponding generating functions $\L_i(\tau)$:
\begin{equation}\label{eq:gf_tau}
\begin{gathered}
  \L_{2m+1}(\tau) = \tau, \;\; \L_{2m}(\tau) = \tau^2 \L_1(\tau), \\
  \L_i(\tau) =  \tau \L_1(\tau)\L_{i+1}(\tau) + \tau \L_{i+2}(\tau) \;\text{ for } 1\le i \le 2m-1.
\end{gathered}  
\end{equation}

\begin{proposition} \label{prop:L_gf}
The following relations hold:
\begin{equation*}
  \L_1(\tau) = \!\sum_{j=0}^{m} \tbinom{m + j}{m -j} \tau^{j+m+1} \L_1^{2j}(\tau) \text{ and }
  \L_2(\tau) = \!\sum_{j=0}^{m-1} \tbinom{m + j}{m-j-1}\tau^{j+m+1} \L_1^{2j+1}(\tau).
\end{equation*}
\end{proposition}

\begin{proof}
 Let $\ell_i = \L_{2m +1 -i}(\tau)$ for $i=0,\dots,2m$. Then \eqref{eq:gf_tau} gives the recurrence relation:
  \begin{gather*}
    \ell_0 = \tau, \quad \ell_1 =  \tau^2 \L_1,\\
    \ell_i = (\tau \L_1)\ell_{i-1} + \tau\ell_{i-2} \;\text{ for } 2\le i \le 2m. 
  \end{gather*}
Using \cite[Prop.~1]{BGW14b} and the fact that $B_{n,j}(c_1,2c_2,0,\dots)=\frac{n!}{j!} \binom{j}{n-j} c_1^{2j-n}c_2^{n-j}$, we find a closed formula for $\ell_n$:
\begin{equation*}
  \ell_n = \sum_{j\ge n/2}^n \binom{j}{n-j} \tau^{j+1}\L_1^{2j-n}.
\end{equation*}
This implies
\begin{equation*}
  \L_1 = \ell_{2m} = \sum_{j=m}^{2m} \binom{j}{2m -j} \tau^{j+1}\L_1^{2j-2m}
  = \sum_{j=0}^{m} \binom{m + j}{m -j}\tau^{j+m+1} \L_1^{2j}, 
\end{equation*}
and similarly,
\begin{equation*}
  \L_2 = \ell_{2m-1} = \sum_{j=0}^{m-1} \binom{m + j}{m-j-1}\tau^{j+m+1} \L_1^{2j+1}.
\end{equation*}
\end{proof}

Since $L_1 = aUb^m$, we have $\L_1(\tau)=\tau^{m+1}\U(\tau)$, and Proposition~\ref{prop:L_gf} gives 
\begin{equation*}
  \tau^{m+1}\U(\tau) = \sum_{j=0}^{m} \binom{m + j}{m -j} \tau^{j+m+1}(\tau^{m+1}\U(\tau))^{2j}.
\end{equation*}
Thus
\begin{equation*}
  \U(\tau) = \sum_{j=0}^{m} \binom{m + j}{m -j} \left(\tau^{2m+3} \U^{2}(\tau)\right)^j.
\end{equation*}
Moreover, since the length of any nonempty word of valuation 0 with slope $\frac{2m+1}{2}$ must be a multiple of $2m+3$, the generating function $\U$ is of the form $1+\sum_{n=1}^\infty u_n \tau^{(2m+3)n}$, where $u_n$ denotes the number of words in $U$ of length $(2m+3)n$. 

Therefore, with the change of variables $t=\tau^{2m+3}$, and denoting the generating function again by $\U$, we obtain
\begin{equation} \label{eq:U(t)}
 \U(t) = 1+ \sum_{j=1}^{m} \binom{m + j}{m -j} t^j \U^{2j}(t).
\end{equation}
For a fixed $m\in\mathbb{N}$, let $\mu_j =  \binom{m+j}{m-j}$ for $j\ge 0$. Note that $\mu_j = 0$ for $j\ge m+1$.

\begin{theorem} \label{thm:un}
If $u_n$ is the coefficient of $t^n$ in $\U(t)$, then
\begin{equation} \label{eq:un}
  u_n = \sum_{k=1}^n \binom{2n}{k-1} \frac{(k-1)!}{n!} B_{n,k}(1!\mu_1,2!\mu_2,\dots).
\end{equation}
\end{theorem}
\begin{proof}
With $X(t)=1+ \sum_{j=1}^{m} \mu_j t^j$, identity \eqref{eq:U(t)} means $\U(t)=X(t\U^2(t))$. Denote the right-hand side of \eqref{eq:un} by $v_n$ and let $\V(t)=1+\sum_{n=1}^\infty v_nt^n$.

\smallskip
We will prove that $\V(t)$ also satisfies $\V(t)=X(t\V^2(t))$, hence $\U(t)=\V(t)$ and so $u_n=v_n$ for every $n$.

\smallskip
First, using the identity (cf.\ \cite[Sec.~11.2, Eqn.~(11.11)]{Charalambides}) 
\[ B_{n, k+1}(z_1,z_2,\dots)  = \sum_{\ell=k}^{n-1} \tbinom{n-1}{\ell} z_{n-\ell} B_{\ell, k}(z_1,z_2,\dots), \] 
we have
\begin{align*}
  n!v_n &= \sum_{k=1}^n \tbinom{2n}{k-1} (k-1)! B_{n,k}(1!\mu_1,2!\mu_2,\dots) \\
  &= \sum_{k=1}^n \tbinom{2n}{k-1} (k-1)! 
  \sum_{\ell=k-1}^{n-1} \tbinom{n-1}{\ell} (n-\ell)! \mu_{n-\ell} B_{\ell,k-1}(1!\mu_1,\dots) \\ 
  &= \sum_{k=0}^{n-1} \tbinom{2n}{k} k! 
  \sum_{\ell=k}^{n-1} \tbinom{n-1}{\ell} (n-\ell)! \mu_{n-\ell} B_{\ell,k}(1!\mu_1,\dots) \\
  &= \sum_{\ell=0}^{n-1} \tbinom{n-1}{\ell} (n-\ell)! \mu_{n-\ell}  
  \sum_{k=0}^{\ell} \tbinom{2n}{k} k! B_{\ell,k}(1!\mu_1,\dots) \\
  &= n!\mu_n + \sum_{\ell=1}^{n-1} \tfrac{n!}{\ell!} \mu_{n-\ell}\, S_{n,\ell}(1!\mu_1,2!\mu_2,\dots),   
\end{align*}
where $S_{n,\ell}(1!\mu_1,2!\mu_2,\dots) = 2(n-\ell)\sum_{k=1}^{\ell} \binom{2n-1}{k-1} (k-1)! B_{\ell,k}(1!\mu_1,2!\mu_2,\dots)$.

\medskip
On the other hand, if $(w_n)$ is defined by $t\V^2(t) = \sum_{n=1}^\infty w_n t^n$, then Fa{\`a} di Bruno's formula (cf.\ \cite[Sec.~3.4, Theorem~A]{Comtet}) implies
\begin{equation*}
 n![t]^n X(t \V^2(t)) = \sum_{\ell=1}^n \ell! \mu_\ell B_{n,\ell}(1!w_1, 2!w_2,\dots).
\end{equation*}
Moreover, by Equation (3$\ell$) in \cite[Sec.~3.3]{Comtet}, and since $w_1=1$, we get
\begin{align*}
 B_{n,\ell}(1!w_1, 2!w_2,\dots) &= \sum_{k=2\ell-n}^\ell \tfrac{n!}{(n-\ell)! k!} B_{n-\ell,\ell-k}(1!w_2,2!w_3,\dots) \\
 &= \sum_{k=0}^{n-\ell} \tfrac{n!}{(n-\ell)! (\ell-k)!} B_{n-\ell,k}(1!w_2,2!w_3,\dots).
\end{align*}
Therefore,
\begin{align*}
 n![t]^n X(t \V^2(t)) 
 &= \sum_{\ell=1}^n \ell! \mu_\ell \sum_{k=0}^{n-\ell} \tfrac{n!}{(n-\ell)! (\ell-k)!} B_{n-\ell,k}(1!w_2,2!w_3,\dots) \\
 &= \sum_{\ell=0}^{n-1} (n-\ell)! \mu_{n-\ell} \sum_{k=0}^{\ell} \tfrac{n!}{\ell! (n-\ell-k)!} B_{\ell,k}(1!w_2,2!w_3,\dots) \\
 &= n!\mu_n + \sum_{\ell=1}^{n-1} \tfrac{n!}{\ell!} \mu_{n-\ell}\, T_{n,\ell}(1!w_2,2!w_3,\dots),
\end{align*}
where $T_{n,\ell}(1!w_2,2!w_3,\dots) = \sum_{k=1}^{\ell} \binom{n-\ell}{k}k! B_{\ell,k}(1!w_2,2!w_3,\dots)$. 

\medskip
Now, using \cite[Sec.~3.5, Thm.~B]{Comtet} we can write $\left(\V^2(t)\right)^{n-\ell}$ and $\V^{2n-2\ell}(t)$ in terms of partial Bell polynomials as follows:
\begin{align*}
  \big(1+w_2t+w_3 t^2+\cdots\big)^{n-\ell} 
  &= 1 + \sum_{\ell=1}^\infty  \sum_{k=1}^\ell \tbinom{n-\ell}{k}\tfrac{k!}{\ell!} B_{\ell,k}(1!w_2,2!w_3,\dots) t^\ell, \\
  \big(1+v_1t+v_2 t^2+\cdots\big)^{2n-2\ell}
  &= 1 + \sum_{\ell=1}^\infty  \sum_{k=1}^\ell \tbinom{2n-2\ell}{k}\tfrac{k!}{\ell!} B_{\ell,k}(1!v_1,2!v_2,\dots) t^\ell.
\end{align*}
This implies $T_{n,\ell}(1!w_2,2!w_3,\dots) = \ell![t^\ell] \left(\V^2(t)\right)^{n-\ell} = \ell![t^\ell] \V^{2n-2\ell}(t)$, and so
\begin{align*}
 T_{n,\ell}(1!w_2,2!w_3,\dots)
 &= \sum_{k=1}^{\ell} \tbinom{2n-2\ell}{k} k! B_{\ell,k}(1!v_1,2!v_2,\dots) \\
 &= 2(n-\ell)\sum_{k=1}^{\ell} \tbinom{2n-2\ell-1}{k-1} (k-1)! B_{\ell,k}(1!v_1,2!v_2,\dots).
\end{align*}

Finally, applying \cite[Theorem~15]{BGW12} with $\lambda=2n-2\ell-1$ on the expression
\begin{equation*}
  \ell! v_\ell = \sum_{k=1}^\ell \tbinom{2\ell}{k-1} (k-1)! B_{\ell,k}(1!\mu_1,2!\mu_2,\dots),
\end{equation*}
we get
\begin{equation*}
 \sum_{k=1}^{\ell} \tbinom{2n-2\ell-1}{k-1} (k-1)! B_{\ell,k}(1!v_1,\dots) = \sum_{k=1}^{\ell} \tbinom{2n-1}{k-1} (k-1)! B_{\ell,k}(1!\mu_1,\dots).
\end{equation*}
Multiplying both sides by $2(n-\ell)$, we finally obtain
\[ T_{n,\ell}(1!w_2,2!w_3,\dots)=S_{n,\ell}(1!\mu_1,2!\mu_2,\dots), \] 
which implies $\V(t)=X(t\V^2(t))$. 
\end{proof}

\bigskip
Let us now address the enumeration of the set $D$ of factor-free Dyck words with slope $\frac{2m+1}{2}$. Let $\D$ denote the generating function of $D$. As discussed in Section~\ref{sec:preliminaries}, we have 
\begin{equation*} 
  D = \varepsilon + L_1^2 b + L_2b.
\end{equation*}
Thus Proposition~\ref{prop:L_gf} together with the factorization $\L_1(\tau)=\tau^{m+1}\U(\tau)$ give
\begin{align*}
  \D &= 1+\tau (\tau^{m+1}\U(\tau))^2 + \sum_{j=0}^{m-1} 
  \tbinom{m + j}{m-j-1}\tau^{j+m+2} (\tau^{m+1}\U(\tau))^{2j+1}(\tau) \\
 &= 1+\tau^{2m+3}\U^2(\tau) + \sum_{j=0}^{m-1} \tbinom{m + j}{m-j-1}\tau^{(2m+3)(j+1)} \U^{2j+1}(\tau) \\
 &= 1+\tau^{2m+3}\U^2(\tau) + \sum_{j=1}^{m} \tbinom{m + j-1}{m-j}\tau^{(2m+3)j} \U^{2j-1}(\tau).
\end{align*}

Setting again $t=\tau^{2m+3}$ and denoting the generating functions with the same letters ($\U$ and $\D$, respectively) but as functions of $t$, we arrive at
\begin{equation} \label{eq:D(t)}
    \D(t)  = 1+ t\U^2(t) +   \sum_{j=1}^{m} \binom{m+j-1}{m-j} t^{j} \U^{2j-1}(t).
\end{equation}

As a consequence of \eqref{eq:D(t)} and Theorem~\ref{thm:un}, we obtain:
\begin{corollary} \label{cor:dn}
If we write $\D(t) = 1 + \sum_{n\ge 1} \theta_n t^n$, then
\begin{equation*}
  \theta_n = \!\! \sum_{\ell=0}^{\min(m, n-1)} \binom{m+\ell+1}{m-\ell} \Delta_{n-\ell-1,\ell},
\end{equation*}
where
\begin{equation*}
  \Delta_{\nu,\ell} = \frac{2\ell+1}{2\nu+2\ell+1}
  \sum_{k=0}^{\nu} \binom{2\nu+2\ell+1}{k} \frac{k!}{\nu!} B_{\nu,k}(1!\mu_1,2!\mu_2,\dots).
\end{equation*}
\end{corollary}
\begin{proof}
We use a similar strategy as for the proof of Theorem~\ref{thm:un}. Here we write the powers of $\U(t)$ in terms of partial Bell polynomials, and then apply \cite[Theorem~15]{BGW12} to rewrite all expressions in terms of the sequence $\mu_1,\mu_2,\dots$. For convenience, let $!\mu$ denote the sequence $(1!\mu_1,2!\mu_2,\dots)$. 

We will work with $\D(t)$ as given in \eqref{eq:D(t)}. First, let us consider $t\U^2(t)$. Using \eqref{eq:un} together with \cite[Sec.~3.5, Theorem~B]{Comtet} and \cite[Theorem~15]{BGW12} with $\lambda=1$, we obtain
\begin{align*}
 [t^n] t\U^2(t) &= \sum_{k=1}^{n-1} 2\tbinom{2n - 1}{k-1} \tfrac{(k-1)!}{(n-1)!} B_{n-1,k}(!\mu) \\
 &= \sum_{k=1}^{n-1}\left(\tbinom{2n - 2}{k-1} + \tfrac{2n+k-2}{2n-1}\tbinom{2n - 1}{k-1}\right) 
 \tfrac{(k-1)!}{(n-1)!}B_{n-1,k}(!\mu) \\
 &= \sum_{k=1}^{n-1}\tfrac{1}{2n-1}\tbinom{2n - 1}{k} \tfrac{k!}{(n-1)!}B_{n-1,k}(!\mu) \\
 &\hspace{2em} + \sum_{k=0}^{n-2} \tfrac{1}{2n-1}\tbinom{2n - 1}{k}\tfrac{k!}{(n-1)!} 
 \big(2(n-1)+k+1\big)B_{n-1,k+1}(!\mu).
\end{align*}
Now, using identities \cite[Sec.~11.2, Eqns.~(11.11) \& (11.12)]{Charalambides} we get
\begin{equation*}
 \big(2(n-1)+k+1\big)B_{n-1,k+1}(!\mu) 
 = \sum_{\ell=1}^{n-k-1} (2\ell+1) \tbinom{n-1}{\ell} \ell!\mu_\ell B_{n-1-\ell, k}(!\mu),
\end{equation*}
and thus $[t^n] t\U^2(t)$ can be written as
\begin{equation*}
 [t^n] t\U^2(t) = \sum_{\ell=0}^{n-1}\tfrac{2\ell+1}{2n-1} \tbinom{m+\ell}{m-\ell}
 \sum_{k=0}^{n-\ell-1} \tbinom{2n - 1}{k}\tfrac{k!}{(n-1-\ell)!} B_{n-\ell-1, k}(!\mu).
\end{equation*}

\medskip
Similarly, but now applying \cite[Theorem~15]{BGW12} with $\lambda=2\ell-2$, we obtain
\begin{equation*}
  [t^n] t^\ell \U^{2\ell - 1}(t) 
  = \tfrac{2\ell-1}{2n-1} \sum_{k=1}^{n-\ell} \tbinom{2n -1}{k} \tfrac{k!}{(n-\ell)!} B_{n-\ell,k}(!\mu).
\end{equation*}
This implies
\begin{align*}
 [t^n] &\sum_{\ell=1}^m \!\tbinom{m+\ell-1}{m-\ell} t^{\ell} \U^{2\ell-1} 
  = \!\sum_{\ell=0}^{m-1} \tfrac{2\ell+1}{2n-1} \tbinom{m+\ell}{m-\ell-1}
  \!\sum_{k=1}^{n-\ell-1}\! \tbinom{2n -1}{k} \tfrac{k!}{(n-\ell-1)!} B_{n-\ell-1,k}(!\mu).
\end{align*}
Finally, the corollary follows by combining the above formulas.
\end{proof}

\section{Examples: Factor-free words with slope 3/2 and 5/2}
\label{sec:examples}

The purpose of this section is to illustrate Theorem~\ref{thm:un} for factor-free words with slope $\frac32$ and $\frac52$. In the case of slope $\frac32$, we recover the results obtained by Duchon \cite[Sec.~6.3]{Duchon} who observed that $(u_n)$ is the sequence of Catalan numbers $(C_n)$. Indeed, by Theorem~\ref{thm:un}, the number of words of length $5n$ in $U_{\frac32}$ is given by
\begin{equation*}
  u_n = \sum_{k=1}^n \tbinom{2n}{k-1} \tfrac{(k-1)!}{n!} B_{n,k}(1,0,\dots) 
  =  \binom{2n}{n-1} \frac{(n-1)!}{n!} = C_n.
\end{equation*}

Moreover, since $\U(t) = 1+ t \U^{2}(t)$ by \eqref{eq:U(t)}, identity \eqref{eq:D(t)} implies
 \begin{equation*}
    \D(t)  = 1+ t\U^2(t) + t\U(t) = \U(t) + t\U(t).
\end{equation*}
Thus the number of factor-free Dyck words with slope $\frac32$ and length $5n$ is given by
\begin{equation*}
  \theta_n = u_n+u_{n-1} = C_n+C_{n-1},
\end{equation*}
as already established in \cite[Proposition~5]{Duchon}.

\bigskip
Let us now look at the set $U$ for the case $m=2$ (slope $\frac52$). 

By Theorem~\ref{thm:un}, the number of words of length $7n$ in $U_{\frac52}$ is given by
\begin{align*}
  u_n &= \sum_{k=1}^n \tbinom{2n}{k-1} \tfrac{(k-1)!}{n!} B_{n,k}(3, 2, 0, \dots) \\
  &= \sum_{k=\lceil n/2\rceil}^n \frac1{k} \binom{2n}{k-1} \binom{k}{n-k} 3^{2k-n} \\
  &= \frac{1}{2n+1} \sum_{k=\lceil n/2\rceil}^n \binom{2n+1}{k} \binom{k}{n-k} 3^{2k-n},
\end{align*}
which gives the numbers {\small 3, 19, 153, 1390, 13581, 139315, 1479855,}\,\dots

Moreover, by Corollary~\ref{cor:dn}, the number of factor-free Dyck words with slope $\frac52$ and length $7n$ is given by
\begin{equation*}
  \theta_n = \!\! \sum_{\ell=0}^{\min(2, n-1)} \binom{\ell+3}{2-\ell} \Delta_{n-\ell-1,\ell},
\end{equation*}
where
\begin{equation*}
 \Delta_{\nu,\ell} = \frac{2\ell+1}{2\nu+2\ell+1} \sum_{k=0}^{\nu} \binom{2\nu+2\ell+1}{k} \binom{k}{\nu-k}3^{2k-\nu}.
\end{equation*}
Hence $\theta_1=3$, $\theta_2=13$, and for $n\ge3$ we have
\begin{equation*}
  \theta_n = \!\! \sum_{\ell=0}^{2} \binom{\ell+3}{2-\ell} \Delta_{n-\ell-1,\ell} 
  = 3\Delta_{n-1,0} + 4\Delta_{n-2,1} + \Delta_{n-3,2}
\end{equation*}
This gives the sequence {\small 3, 13, 94, 810, 7667, 76998, 805560,} \dots, \cite[A274052]{Sloane}.

\subsection*{Combinatorial interpretation of $U_{\frac52}$}
As discussed in Example~\ref{ex:slope5/2}, the language $U=U_{\frac52}$ has the grammar
\begin{equation*}
  U = \varepsilon + aUbbb aUbbb aUbbb aUb + baUbbbaUb + aUbbbbaUb + aUbbbaUbb.
\end{equation*}
This suggests a natural bijection to rooted planar trees. Specifically:
\begin{quote}
The words of length $7n$ in $U_{\frac52}$ are in one-to-one correspondence with rooted trees with $2n$ edges having nonleaf nodes of outdegrees 2 or 4, where nodes of outdegree 2 may be colored in three different ways. 
\end{quote}
We finish this section with a discussion of this bijection. 

Note that $babbbab$, $abbbbab$, $abbbabb$, and $abbbabbbabbbab$, are the basic words needed to build all other words in $U$. We identify these building blocks with colored trees as follows:

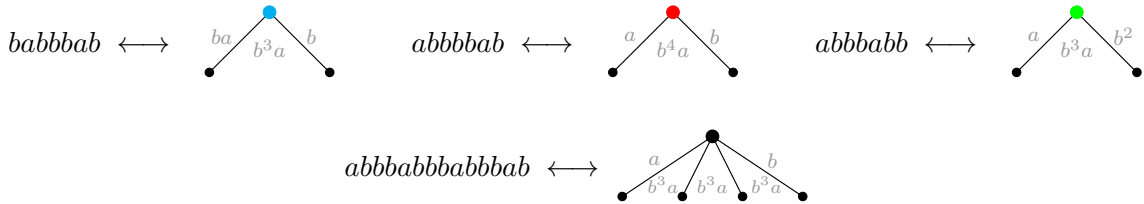
\begin{center}
\begin{tikzpicture}[scale=0.8]
\node at (-3,-0.5) {$babbbab\; \longleftrightarrow$};
\draw[fill] (0,0)-- (-1,-1) circle(2pt); \draw[fill] (0,0) -- (1,-1) circle(2pt); \draw[fill,cyan] (0,0) circle(3pt);
\node[left=1pt,gray!80] at (-0.4,-0.4) {\scriptsize $ba$};
\node[gray!80] at (0,-0.6) {\scriptsize $b^3a$};
\node[right=1pt,gray!80] at (0.4,-0.4) {\scriptsize $b$};
\end{tikzpicture}
\qquad
\begin{tikzpicture}[scale=0.8]
\node at (-3,-0.5) {$abbbbab\; \longleftrightarrow$};
\draw[fill] (0,0)-- (-1,-1) circle(2pt); \draw[fill] (0,0) -- (1,-1) circle(2pt); \draw[fill,red] (0,0) circle(3pt);
\node[left=1pt,gray!80] at (-0.4,-0.4) {\scriptsize $a$};
\node[gray!80] at (0,-0.6) {\scriptsize $b^4a$};
\node[right=1pt,gray!80] at (0.4,-0.4) {\scriptsize $b$};
\end{tikzpicture}
\qquad
\begin{tikzpicture}[scale=0.8]
\node at (-3,-0.5) {$abbbabb\; \longleftrightarrow$};
\draw[fill] (0,0)-- (-1,-1) circle(2pt); \draw[fill] (0,0) -- (1,-1) circle(2pt); \draw[fill,green] (0,0) circle(3pt);
\node[left=1pt,gray!80] at (-0.4,-0.4) {\scriptsize $a$};
\node[gray!80] at (0,-0.6) {\scriptsize $b^3a$};
\node[right=1pt,gray!80] at (0.4,-0.4) {\scriptsize $b^2$};
\end{tikzpicture}

\vspace{4ex}
\begin{tikzpicture}[scale=0.8]
\node at (-4,-0.5) {$abbbabbbabbbab\; \longleftrightarrow$};
\draw[fill] (0,0) circle(3pt) -- (-1.5,-1) circle(2pt); 
\draw[fill] (0,0) -- (-0.5,-1) circle(2pt); \draw[fill] (0,0) -- (0.5,-1) circle(2pt); \draw[fill] (0,0) -- (1.5,-1) circle(2pt);
\node[left=1pt,gray!80] at (-0.65,-0.4) {\scriptsize $a$};
\node[gray!80] at (-0.82,-0.8) {\tiny $b^3a$};
\node[gray!80] at (0,-0.8) {\tiny $b^3a$};
\node[gray!80] at (0.9,-0.8) {\tiny $b^3a$};
\node[right=1pt,gray!80] at (0.7,-0.4) {\scriptsize $b$};
\end{tikzpicture}
\end{center}

The grammar of $U$ implies that any word of length $7n$ can be formed by inserting (right after an $a$):
\begin{itemize}
\itemsep0pt
\item[$(i)$] a word of length $7(n-1)$ into one of the three words of length 7, 
\item[$(ii)$] or a word of length $7(n-2)$ into the word $abbbabbbabbbab$.
\end{itemize}
Therefore, an `$a$' or a `$ba$' that is not part of a $bbba$ string is always followed by one element of the set of subwords $\{a, ba, bbba, bbbba\}$.

For any word $u\in U$ of length $7n$, we will construct a tree with the properties stated in the bijection. To this end, traverse the word from left to right and do the following (labeling the left edges accordingly): 
\begin{itemize}
\itemsep0pt
\item[$\circ$] For every $a$ or $ba$ that is not part of a $bbba$ string, draw a left edge and move to the leaf just created. If the edge is labeled with `$ba$', color the parent node blue.
\item[$\circ$] For every $bbba$ that is not part of a $bbbba$ string, draw a right edge from the parent node and move to the leaf just created.
\item[$\circ$] For every $bbbba$, there are two possible steps:
\begin{itemize}
\itemsep0pt
\item If the edge created last was a left edge, draw a right edge from the parent node, move to the leaf just created, and color the parent node red.
\item Otherwise, draw a right edge from the {\em grandparent} node and move to the leaf just created.
\end{itemize}
\item[$\circ$] For every $bb$ that is not part of a $bbba$ or a $bbbba$ string, move to the grandparent node unless the current node is a leaf of a binary subtree whose left edge is labeled with an `$a$'. In the latter case, just move to the parent node and color it green.
\item[$\circ$] For every $b$ that doesn't fall into any of the previous cases, move to the parent node.
\end{itemize}
Since every appearance of $a$ is responsible for the creation of an edge, a word in $U$ of length $7n$ corresponds to a tree with $2n$ edges.

\smallskip
The reverse algorithm is clear. Given a colored tree with $2n$ edges, label each of the four building subtrees according to the identification given above. Then traverse the tree counterclockwise, starting at the root, and record the labels writing the letters from left to right. Depending on the color of the nodes, write $a$ or $ba$ when traveling down along a left edge, write $bbba$ or $bbbba$ when traveling up and down between adjacent edges, and write $b$ or $bb$ when traveling up along a right edge. The resulting word has length $7n$ and belongs to $U_{\frac52}$.

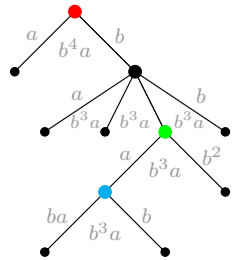
\begin{figure}[!ht]
\begin{center}
\begin{tikzpicture}[scale=0.8]
\draw[fill] (-1,1)-- (-2,0) circle(2pt); \draw[fill] (-1,1) -- (0,0); \draw[fill,red] (-1,1) circle(3pt);
\draw[fill] (0,0) circle(3pt) -- (-1.5,-1) circle(2pt); 
\draw[fill] (0,0) -- (-0.5,-1) circle(2pt); \draw[fill] (0,0) -- (0.5,-1); \draw[fill] (0,0) -- (1.5,-1) circle(2pt);
\draw[fill] (0.5,-1)-- (-0.5,-2) circle(2pt); \draw[fill] (0.5,-1) -- (1.5,-2) circle(2pt); 
\draw[fill,green] (0.5,-1) circle(3pt);
\draw[fill] (-0.5,-2)-- (-1.5,-3) circle(2pt); \draw[fill] (-0.5,-2) -- (0.5,-3) circle(2pt); 
\draw[fill,cyan] (-0.5,-2) circle(3pt);
\node[left=1pt,gray!80] at (-1.4,0.6) {\scriptsize $a$};
\node[gray!80] at (-1,0.4) {\scriptsize $b^4a$};
\node[right=1pt,gray!80] at (-0.55,0.6) {\scriptsize $b$};
\node[left=1pt,gray!80] at (-0.65,-0.4) {\scriptsize $a$};
\node[gray!80] at (-0.82,-0.8) {\tiny $b^3a$};
\node[gray!80] at (0,-0.8) {\tiny $b^3a$};
\node[gray!80] at (0.9,-0.8) {\tiny $b^3a$};
\node[right=1pt,gray!80] at (0.8,-0.4) {\scriptsize $b$};
\node[left=1pt,gray!80] at (0.15,-1.4) {\scriptsize $a$};
\node[gray!80] at (0.5,-1.6) {\scriptsize $b^3a$};
\node[right=1pt,gray!80] at (0.9,-1.4) {\scriptsize $b^2$};
\node[left=1pt,gray!80] at (-0.9,-2.4) {\scriptsize $ba$};
\node[gray!80] at (-0.5,-2.65) {\scriptsize $b^3a$};
\node[right=1pt,gray!80] at (-0.1,-2.4) {\scriptsize $b$};
\end{tikzpicture}
\end{center}
\caption{Tree representation of $abbbbaabbbabbbaababbbabbbbabbbbbabb$.}
\end{figure}

\section{Further remarks and applications}

\subsection{Slope \texorpdfstring{$\frac72$}{7/2}}
If $m=3$, we get $\mu_1=\binom{4}{2}=6$, $\mu_2=\binom{5}{1}=5$, and $\mu_3=\binom{6}{0}=1$. Thus there are 12 basic words that may be used as building blocks to create all words in $U_{\frac72}$:
\begin{center}
\begin{tabular}{p{3cm}c}
 $abbbbabbb$ & $abbbbabbbbabbbbabb$ \\
 $abbbbbabb$ & $abbbbabbbbabbbbbab$ \\
 $abbbbbbab$ & $abbbbabbbbbabbbbab$ \\
 $babbbbabb$ & $abbbbbabbbbabbbbab$ \\
 $babbbbbab$ & $babbbbabbbbabbbbab$ \\
 $bbabbbbab$ & $abbbbabbbbabbbbabbbbabbbbab$ 
\end{tabular}
\end{center}

\subsection{Connection to colored Dyck paths}
Combining equation \eqref{eq:un} with \cite[Theorem~3.5]{BGMW}, we conclude that there is a one-to-one correspondence between the words in $U_{\frac{2m+1}{2}}$ of length $(2m+3)n$ and the set of Dyck words of semilength $2n$ created from strings of the form `$d$' and `$u^{2j}d$' for $j=1,\ldots,\min(m,n)$, such that each maximal $2j$-ascent may be colored in $\mu_j$ different ways. 
For example, for $m=2$ this means that there is a bijection between the words in $U_{\frac52}$ of length $7n$ and the set of Dyck words of semilength $2n$ with ascents of length 2 or 4 and such that each ascent of length 2 may be colored in three different ways. 

\subsection{Factor-free words}
The importance of factor-free Dyck words relies on the fact that, as shown in \cite[Section~4]{Duchon}, words in a generalized Dyck language can be obtained uniquely by inserting words of the language into factor-free words of the same language. In particular, in the case of a two-letter alphabet, for which words correspond to rational Dyck paths, Duchon \cite[Theorem~9]{Duchon} established a direct generating function connection between rational Dyck paths and their subset of factor-free elements. 

As discussed in this paper, when the slope of the Dyck words is $\frac{2m+1}{2}$, understanding the auxiliary language $U$ introduced in Section~\ref{sec:preliminaries} suffices to generate and enumerate the set of corresponding factor-free Dyck words with same slope. An explicit enumeration formula is given in \eqref{eq:un}. The language $U$ turned out to have interesting combinatorial properties, which we have illustrated for the cases of slope $\frac32$ and $\frac52$.

\subsection{Cross-bifix-free codes}
By definition, if $w$ is a factor-free Dyck word in $\Dyck_{\A,h}$, then for any representation $w=w_1w_2$ with nonempty subwords $w_1$ and $w_2$, we must have $h(w_1)>0$ and $h(w_2)<0$. Therefore, no prefix of any length of any factor-free word is the suffix of any other factor-free word. This means that any set of factor-free words in $\Dyck_{\A,h}$ is a cross-bifix-free set, and in the case of slope $\frac{2m+1}{2}$, we can generate these sets using the auxiliary language $U$ and the derivation rules \eqref{eq:U}.

For example, the set of factor-free Dyck words over the alphabet $\{0,1\}$ with $h(0)=3$ and $h(1)=-2$ (slope $\frac32$) gives a cross-bifix-free set (non-overlapping code) of binary words with variable length $\equiv 0$ modulo $5$. Here are the words in $D_{\frac32}$ of length 5, 10, and 15:
\begin{align*}
 & 00111, 01011, \\ 
 & 0100110111, 0001101111, 0011011011, \\
 & 010011001101111, 010001101110111, 001101100110111, 000110011011111, \\
 & 001100110111011, 000011011101111, 000110111011011,
\end{align*}
and here are the nineteen words in $D_{\frac32}$ of length 20:
\begin{align*}
& 01001100110011011111, 01001100011011101111, 01000110111001101111, \\
& 01000110011011110111, 01000011011101110111, 00110110011001101111, \\
& 00110110001101110111, 00110011011100110111, 00011011101100110111, \\
& 00011001100110111111, 00110011001101111011, 00011000110111011111, \\
& 00110001101110111011, 00001101110011011111, 00011011100110111011, \\
& 00001100110111101111, 00011001101111011011, 00000110111011101111, \\
& 00001101110111011011.
\end{align*}

\medskip
Of course, similar sets can be constructed by means of $U$ for binary Dyck words with slope $\frac{2m+1}{2}$. However, if we require words in the set to have a fixed length (as it is customary in the literature, see e.g.\ \cite{BPP12,Blackburn}), the codes obtained with these slopes may be too restrictive. Nonetheless, rational lattice paths of other slopes may be used to produce larger codes along the lines of those given by Bilotta et al.\ \cite{BPP12}. For instance, there is a straightforward bijection between the set $CBFS_2(2m+1)$ constructed in op.\ cit.\ and the set of rational Dyck paths from $(0,0)$ to $(m,m+1)$.


\end{document}